\documentclass[11pt, a4paper]{article}
\usepackage[]{amsmath,amsxtra,amssymb,latexsym, amscd,amsthm, amsfonts}
\usepackage{hyperref}
\usepackage{tikz}
\usetikzlibrary{positioning}
\title{ Integer solutions of $a^2+ab+b^2=7^n$}
\author{Nguyen Tho Tung\\
Department of Mathematics-Mechanics-Informatics\\
Vietnam National University, Hanoi, Viet Nam\\
\\ Le Tien Nam\\
Department of Mathematics-Mechanics-Informatics\\
Vietnam National University, Hanoi, Viet Nam\\}
\date{}
 
\newtheorem{Pro}{Proposition}

\newtheorem{cor}{Corollary}
\theoremstyle{definition}

\newtheorem{cl}{Claim}
\newtheorem{rem}{Remark}
\DeclareMathOperator{\R}{\mathbb{R}}

\DeclareMathOperator{\Z}{\mathbb{Z}}
\DeclareMathOperator{\N}{\mathbb{N}}

\begin{document}
\maketitle
\begin{abstract}
In this article we will show $2$ different proofs for the fact that there  exist relatively prime positive integers $a,b$ such that: \[a^2+ab+b^2=7^n.\]
\end{abstract}
\section{The first proof-by Nguyen Tho Tung} 
For the first proof, we use the ring of Eisenstein integers to solve the above problem.

\begin{Pro}
For each positive integer $n$, there exists two positive relatively prime integers $a,b$ such that:
\begin{equation}
 a^2+ab+b^2=7^n.  
\end{equation}
\end{Pro}
%
%

\begin{proof}
Let $\omega=\frac{1+\sqrt{-3}}{2}$. Consider the Integral Domain  $\Z[\omega] = \{a + b\omega \mid a,b \in \Z\}$. This Integral Domain is a Euclidean Domain with the norm given by 
\[ N(a+b\omega)=(a+b\omega) \overline{ (a+b\omega)} = a^2+ab+b^2.\]
\[(a+b\omega)\overline{(a + b\omega)} = (\frac{2a + b + b\sqrt{-3}}{2})(\frac{2a + b - b\sqrt{-3}}{2}) = a^2 + ab + b^2 \] 
We first note that $N(2 + \omega) =7$ and hence $N((2+\omega)^n) = 7^n$.
Our goal is to show that for all $n \in \N$ we can find positive integers $a_n, b_n$ such that:
\[ a_n+b_n \omega =(2+\omega)^n.\] 
Note that this imples that: \[N(a_n + b_n\omega) = a_n^2 + a_nb_n + b_n^2 = N((2+\omega)^n)) = 7^n\]

We have:
\[ a_{n+1}+b_{n+1} \omega=(a_n+b_n\omega)(2+\omega)=(2a_n-b_n)+(a_n+3b_n)\omega .\]
here we use the fact that $\omega^2=\omega-1$. By this formula, we have 
\[ \begin{cases} a_{n+1}=2a_n-b_n \\ b_{n+1}=a_n+3b_n. \end{cases} \] 
with initial conditions:  $a_0=1,b_0=0, a_1=2, b_1=1$. 
\vskip 5pt 
The above relation gives us
\begin{align*}
a_{n+2}&=2a_{n+1}-b_{n+1} =2a_{n+1}-(a_n+3b_n) \\
&=2a_{n+1}-a_n-3(2a_n-a_{n+1})\\
&=5a_{n+1}-7a_n ,\quad \forall n\geq 1. 
\end{align*}
Similarly we also have 
\[ b_{n+2}=5b_{n+1}-7b_n.\] 
Taking the sequence modulo 7 we have 
\[ a_{n+1} \equiv 5a_n \bmod 7, \forall n\geq 1. \] 
As $a_1=2$, from the above relation, we can conclude that $7 \nmid a_n$ for all $n$. In addition, we have
\[ a_n^2+a_nb_n+b_n^2=7^n.\] 
Therefore, $a_n$ and $b_n$ are relatively prime. Otherwise, $7$ would be a divisor of $a_n$. 
\vskip 5pt
It remains to show that we can always find positive integers $a_n, b_n$. In case both are negative then $|a_n|^2 + |a_n||b_n| + |b_n|^2 = 7^n$. If one is negative and the other is positive, since they are relatively prime  we may assume without loss of generality that $b_n > |a_n|$. It is now easy to check that:
\[(b_n - |a_n|)^2 + |a_n|(b_n - |a_n|) + |a_n|^2 = b_n^2 + a_nb_n + a_n^2 = 7^n\]
so we obtain the positive solution \[(b'_n, a'_n) = (b_n - |a_n|, |a_n|).\]

\end{proof}

We will illustrate some numerical data. Let us denote by  $a_n,b_n$  the integer solution of $(1)$ defined above and $(A_n, B_n)$ the positive pair associated with $(a_n,b_n)$.  

\begin{center}
\begin{tabular}{l*{6}{c}r}
n                 & $a_n$ \quad  & $b_n$ \quad   & $A_n$ \quad  & $B_n$\\
\hline
0   			  & 1		& 0     & 1		&  0 				 \\
             
1                & 2		&1		&2		&  1 \\  
2                 & 3		&5		&3		&5   \\
3				  & 1		&18     & 1     &18 \\	
4 				  & -16     &55 	& 39 &16 \\
5				&-87		& 149   & 62 &87 \\ 	
6     			& -323      &360    & 37 & 323 \\ 
\end{tabular}
\end{center}
\begin{cor}
The equation $a^2 + ab + b^2 = 7^{2n}$ has at least $n$ distinct positive solutions $(a,b)$.
\end{cor}
\begin{proof}
\begin{enumerate}
\item \notag
\item $3^2 + 3 \times 5 + 5^2 = 7^2$. 
\item Assume by induction that $a^2 + ab + b^2 = 7^{2n}$ has $n$ disitnct solutions $\{(a_i, b_i), i = 1,\ldots , n\}$.
\item Then $\{(7a_i, 7b_i)\}$ plus the relatively prime pair $(a,b)$ that exists by proposition 1, give us $n+1$ positive solutions.
\end{enumerate}
\end{proof}
\begin{rem}
The proof also holds  for all primes $p$ such that $p\equiv 1 \bmod 6$. 
\end{rem}
\section{The Second Proof- by Le Tien Nam}
This proof is  elementary, easily inderstood by high school students. 

\noindent We restate the proposition and prove it by induction.
\begin{cl} The equation $a^2+b^2+ab=7^n$ has at least one positive integer solution $(a,b)$ such that $\gcd (a,7)=1$
\end{cl}
\begin{proof} 
It is easy to see that for $n=1,2, \; (1,2)\;  and  \;  (3,5)$ solve the equation.\\
Assume that the claim is true for $n$, that is there are relatively prime positive integers $a,b$ satisfying the claim.
We shall prove that the claim is also true for $n+1$.\\
Without loss of generality, we can assume that $a<b$. Consider the following three pairs of numbers:
\begin{itemize}
\item $(c_1,d_1)=(2b-a,3a+b)$
\begin{align}
\Rightarrow c_1^2+d_1^2+c_1d_1 &=(2b-a)^2+(3a+b)^2+(2b-a)(3a+b) \notag \\
&=a^2+4b^2-4ab+9a^2+b^2+6ab-3a^2+2b^2+5ab\notag\\
&=7(a^2+b^2+ab)\notag\\
&=7^{n+1}.\notag
\end{align}
\item $(c_2,d_2)=(b-2a,3a+2b)$
\begin{align}
\Rightarrow c_2^2+d_2^2+c_2d_2 &=(b-2a)^2+(3a+2b)^2+(b-2a)(3a+2b) \notag \\
&=7(a^2+b^2+ab)\notag\\
&=7^{n+1}.\notag
\end{align}
\item $(c_3,d_3)=(2a-b,a+3b)$
\begin{align}
\Rightarrow c_3^2+d_3^2+c_3d_3 &=(2a-b)^2+(a+3b)^2+(2a-b)(a+3b) \notag \\
&=7(a^2+b^2+ab)\notag\\
&=7^{n+1}.\notag
\end{align}
\end{itemize}
By definition we have 
$c_1,d_1,d_2,d_3>0$. Thus,
\begin{itemize}
\item If $\gcd(c_1,7)=1 \Rightarrow (c_1,d_1)$ satisfies the claim for $n+1$.
\item If $\gcd(c_1,7)\ne1 \Rightarrow 7|c_1 \Rightarrow7|2b-a$\\
Note that if $7 | 2a - b$ then $7 | 2b-a + 4a - 2b = 3a$ contradicting the assumption that $gcd(a,7) = 1$, so $\gcd(2a-b,7)=1$
\begin{itemize}
\item If $2a-b<0$, then $(c_2,d_2)$ satisfies the claim for $n+1$.
\item If $2a-b>0$, then $(c_3,d_3)$ satisfies the claim for $n+1$.
\end{itemize}
\end{itemize}
Hence, the claim is true for $n+1$ and by the principle of mathematical induction the claim is true for every positive integer $n$. 
\end{proof}
By a similar argument we can also prove the following generalization: 
\begin{Pro}
For every $r$ such that there exists $(a_0,b_0)$  a pair of relatively prime positive integers such that $r=a_0^2+b_0^2+a_0b_0$ the equation: \[a^2+b^2+ab=r^n\] has at least one positive integer solution $(a_n,b_n)$ such that $\gcd (a_n,b_n)=1$
\end{Pro}
Step up the work, we have a nice proposition for the equation $x^2+y^2+xy=z^2$.
\begin{Pro}
The set of all positive integer solutions of the equation $x^2+y^2+xy=z^2$ is $M\bigcup N$; in which:
\begin{align}
M=& \Big((b^2-a^2,a^2+2ab,a^2+b^2+ab)\Big|(a,b)\in \N,\gcd(a,b)=1\Big)\notag\\
N=& \left(\left(\dfrac{b^2-a^2}{3},\dfrac{a^2+2ab}{3},\dfrac{a^2+b^2+ab}{3}\right)\bigg|(a,b)\in \N,a=b\text{ mod } 3,\gcd(a,b)=1\right)\notag
\end{align}
\end{Pro}
These results are helpful in constructing integer distance geometric graphs.  We provide here two examples. See \textit{"Constructing integer distance graphs"} for more details.

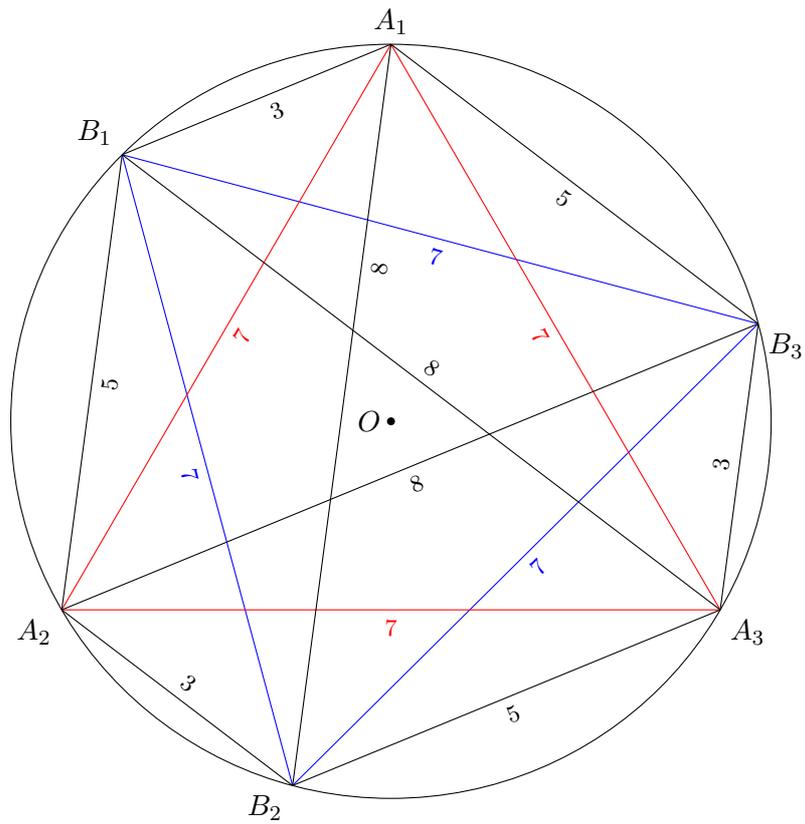
\begin{figure}
\centering

\begin{tikzpicture}
\def\R{5} 
\pgfmathsetmacro{\r}{\R/2} 
\path
(135:\R) coordinate (B1) node[above left]{$B_1$}
(90:\R) coordinate (A1) node[above]{$A_1$}

(255:\R) coordinate (B2) node[below left]{$B_2$}
(210:\R) coordinate (A2) node[below left]{$A_2$}

(15:\R) coordinate (B3) node[below right]{$B_3$}
(-30:\R) coordinate (A3) node[below right]{$A_3$};

\draw (0,0) circle(\R);


\draw [blue] (B1)--(B2) node[midway, below, sloped, pos=0.5] {\footnotesize $7$ };
\draw [blue] (B2)--(B3) node[midway, below, sloped, pos=0.5] {\footnotesize $7$ };
\draw [blue] (B3)--(B1) node[midway, below, sloped, pos=0.5] {\footnotesize $7$ };

\draw [red](A1)--(A2) node[midway, below, sloped, pos=0.5] {\footnotesize $7$};
\draw [red] (A2)--(A3) node[midway, below, sloped, pos=0.5] {\footnotesize $7$};
\draw [red] (A3)--(A1) node[midway, below, sloped, pos=0.5] {\footnotesize $7$};

\draw [-] (B1) -- (A1) node[midway, below, sloped, pos=0.55] {\footnotesize $3$};
\draw [-] (B1) -- (A2) node[midway,sloped, below] {\footnotesize $5$};

\draw [-] (B2) -- (A2) node[midway, sloped, above, pos=0.5] {\footnotesize $3$};
\draw [-] (B2) -- (A3) node[midway, below, sloped] {\footnotesize $5$};

\draw [-] (B3) -- (A3) node[midway, sloped, above, pos=0.5] {\footnotesize $3$};
\draw [-] (B3) -- (A1) node[midway, below, sloped] {\footnotesize $5$};

\draw (A1)--(B2) node[midway, below, sloped, pos=0.3] {\footnotesize $8$};
\draw (A2)--(B3) node[midway, below, sloped, pos=0.5] {\footnotesize $8$};
\draw (A3)--(B1) node[midway, above, sloped, pos=0.5] {\footnotesize $8$};

\fill[] (0,0) circle(1.5pt) node[left] {$O$};

\end{tikzpicture}

\caption{Embedding $K_{2,2,2}$ and the 5-wheel in the integral $\R^2$. The $5$-cycle $A_1B_1A_2A_3B_3$ is a pentagon with odd length edges and $B_2$ is connected by odd length edges to all of them, except to $A_1$ where $|B_2A_1|=8$ by Ptolemy's theorem.} \label{fig:M1}

\end{figure}

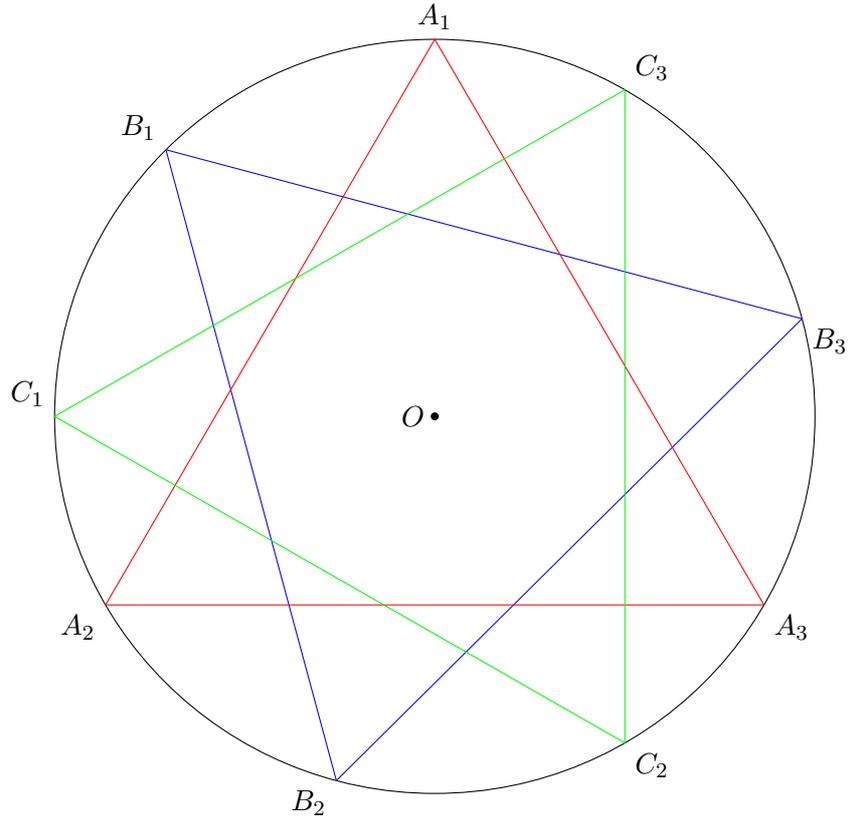
\begin{figure}
\centering
\begin{tikzpicture}
\def\R{5} 
\pgfmathsetmacro{\r}{\R/2} 
\path
(135:\R) coordinate (B1) node[above left]{$B_1$}
(90:\R) coordinate (A1) node[above]{$A_1$}

(255:\R) coordinate (B2) node[below left]{$B_2$}
(210:\R) coordinate (A2) node[below left]{$A_2$}

(15:\R) coordinate (B3) node[below right]{$B_3$}
(-30:\R) coordinate (A3) node[below right]{$A_3$}
(60:\R) coordinate (C3) node[above right]{$C_3$}
(180:\R) coordinate (C1) node[above left]{$C_1$}
(-60:\R) coordinate (C2) node[below right]{$C_2$}
;

\draw (0,0) circle(\R);


\draw [blue] (B1)--(B2) node[midway, below, sloped, pos=0.5]{};
\draw [blue] (B2)--(B3) node[midway, below, sloped, pos=0.5]{};
\draw [blue] (B3)--(B1) node[midway, below, sloped, pos=0.5]{};

\draw [red](A1)--(A2) node[midway, below, sloped, pos=0.5]{}; 
\draw [red](A2)--(A3) node[midway, below, sloped, pos=0.5]{}; 
\draw [red](A3)--(A1) node[midway, below, sloped, pos=0.5]{};

\draw [green] (C1)--(C2) node[midway, below, sloped, pos=0.5]{};
\draw [green] (C2)--(C3) node[midway, below, sloped, pos=0.5]{};
\draw [green] (C3)--(C1) node[midway, below, sloped, pos=0.5]{};

\fill[] (0,0) circle(1.5pt) node[left] {$O$};

\end{tikzpicture}
\caption{Embedding $K_{3,3,3}$ and the 5 and 7-wheels in the integral $\R^2$. The red and blue triangle are the blown of the first picture by a factor of $7$. The green circle is constructed by the requirement that $A1C_3=A_3C_2=A_2C_1=16, C_3A_3=A_2C_2=A_1C_1=39. $ By Ptolemy's theorem we can see that $B_3C_3=B_2C_2=B_1C_1=21$. Additionally, $A_1B_2=A_2B_3=A_3B_1=56$,  $B_2C_1=B_1C_3=B_3C_2=35$, and $B_1C_2=B_2C_3=C_1B_3=55.$} \label{fig:M1}

\end{figure}

\end{document}